\documentclass[a4paper,12pt]{amsart}

\usepackage[utf8]{inputenc}

\usepackage[english]{babel}
\usepackage{comment}

\usepackage{amsmath}
\usepackage{amssymb}
\usepackage{amsthm}

\usepackage{graphicx}

\usepackage{hyperref}
\usepackage[capitalize]{cleveref}

\usepackage[stable]{footmisc}  

\usepackage{tikz}
\usetikzlibrary{matrix}
\usetikzlibrary{positioning}
\usetikzlibrary{cd}

\usepackage{mfirstuc}  
\usepackage{listings}  
\usepackage{parskip}

\mathchardef\mhyphen="2D

\newtheorem{lemma}{Lemma}[section]
\newtheorem{proposition}[lemma]{Proposition}
\newtheorem{theorem}[lemma]{Theorem}
\newtheorem{corollary}[lemma]{Corollary}

\theoremstyle{definition}
\newtheorem{definition}[lemma]{Definition}

\newtheorem{remark}[lemma]{Remark}

\let\phi\varphi
\let\varphi\phi
\let\subsetold\subset
\let\subset\subseteq

\renewcommand{\O}{\mathcal{O}}

\newcommand{\spec}{\operatorname{Spec}}

\newcommand{\tensor}{\otimes}

\newcommand{\CC}{\mathbb{C}}

\newcommand{\ZZ}{\mathbb{Z}}
\newcommand{\QQ}{\mathbb{Q}}

\newcommand{\PP}{\mathbb{P}}

\newcommand{\inv}{^{-1}}

\newcommand{\Hom}{\operatorname{Hom}}

\newcommand{\Gal}{\operatorname{Gal}}
\newcommand{\GL}{\operatorname{GL}}
\newcommand{\SL}{\operatorname{SL}}
\newcommand{\Res}{\operatorname{Res}}

\newcommand{\bmat}{\begin{pmatrix}}
\newcommand{\emat}{\end{pmatrix}}

\usepackage{doi}
\usepackage{dsfont, libertine}
\newcommand{\Frob}{\operatorname{Frob}}
\newcommand{\length}{\operatorname{length}}

\newcommand{\Spec}{\operatorname{Spec}}

\newcommand{\OO}{\O}
\newcommand{\OK}{\O_K}
\newcommand{\OKS}{\OK[S\inv]}
\newcommand{\Kbar}{\overline{K}}

\newcommand{\Hdr}{H_{\textrm{dR}}}

\title{Siegel's theorem via the Lawrence--Venkatesh method}
\author{Marc Paul Noordman}
\address{Bernoulli Institute, University of Groningen, The Netherlands}
\date{January 18, 2021}

\begin{document}
\maketitle

\begin{abstract}
    In the recent paper \cite{Lawrence--Venkatesh}, B. Lawrence and A. Venkatesh develop a method of proving finiteness theorems in arithmetic geometry by studying the geometry of families over a base variety. Their results include a new proof of both the $S$-unit theorem and Faltings' theorem, obtained by constructing and studying suitable abelian-by-finite families over $\PP^1\setminus\{0,1,\infty\}$ and over an arbitrary curve of genus $\geq 2$ respectively. In this paper, we apply this strategy to reprove Siegel's theorem: we construct an abelian-by-finite family on a punctured elliptic curve to prove finiteness of $S$-integral points on elliptic curves.  
\end{abstract}

\setlength{\parskip}{0ex plus 0.5ex minus 0.5ex}
\tableofcontents
\setlength{\parskip}{1ex plus 0.5ex minus 0.5ex}

\section{Introduction}
Let $K$ be a number field and $E$ an elliptic curve over $K$ given by a Weierstrass equation $y^2 = x^3 + ax + b$. Let $S$ be any finite set of places of $K$ including those where $a$ and $b$ have negative valuation and those dividing the discriminant of this Weierstrass model. An $S$-integral point of $E$ is then a solution $(x,y)$ of this Weierstrass equation with $x, y \in \OKS$. Siegel's theorem states that $E$ has only finitely many $S$-integral points (see \cite[Chapter 3]{Corvaja} for a modern treatment of this theorem, including several proofs).

In the recent paper \cite{Lawrence--Venkatesh}, B. Lawrence and A. Venkatesh develop a method of proving finiteness theorems in arithmetic geometry, by studying the geometry of families over a base and the associated complex-analytic and $p$-adic period mappings. They apply this method to prove or reprove several results in arithmetic geometry, including reproofs of the $S$-unit theorem (finiteness of the set of $S$-integral points on $\PP^1\setminus\{0,1,\infty\}$) and of Faltings' theorem (finiteness of the set of rational points on smooth projective curves of genus $\geq 2$). This new approach has generated much interest. In \cite{Lawrence-Sawin} this method is used to prove a Shafarevich theorem for hypersurfaces of abelian varieties. Uniformity aspects of the Lawrence--Venkatesh method are analyzed in \cite{Nasserden-Xiao}, and effectivity aspects in  \cite{ManyAuthors}. The latter work moreover compares the Lawrence--Venkatesh method to the Kim--Chabauty approach to (effective) finiteness of rational points on curves. Finally, we mention the work of \cite{Javanpeykar-Litt}, who show that in the context of the Lawrence--Venkatesh method, one can often extend results about finiteness or non-Zariski denseness of sets of points over number rings to the same results for points over more general finitely generated rings. 

The goal of the present paper is to show that Siegel's theorem admits a proof via the Lawrence--Venkatesh method as well. Note that since \cite{Lawrence--Venkatesh} already handles the case of $\PP^1\setminus\{0,1,\infty\}$ and of smooth curves of genus $\geq 2$, this is the only remaining case left in dimension 1.\footnote{After uploading this work to arXiv, we have been made aware of H. Liu's thesis \cite{LiuMasterThesis}, which also handles this case.} We construct an abelian-by-finite family over a punctured elliptic curve and we show that it has the correct properties for the Lawrence--Venkatesh method to succeed. 

The paper consists of two parts. In the first section, we revisit briefly the Lawrence--Venkatesh method. We formulate a theorem (\cref{main-theorem}) which summarizes the result of this approach, in the case of an abelian-by-finite family. This theorem may be considered a ``black box'', in that one can apply this theorem without knowing the intricate and sometimes technically demanding techniques underlying \cite{Lawrence--Venkatesh}; as such we hope that articulating this theorem explicitly may help increase the reach and usability of this new method. In the second section, we show how to apply the Lawrence--Venkatesh method to prove Siegel's theorem. We construct a suitable abelian-by-finite family of a punctured elliptic curve, and show that it satisfies the conditions of the black box theorem. As is often the case, the main difficulty is showing that the monodromy is large enough.

\subsection*{Notation}
In the remainder of this paper we will use the following notation throughout. 

Let $K$ be a number field with ring of integers $\OO_K$, and $S$ a finite set of finite places of $K$. We will assume that $S$ contains all places of $K$ that are ramified over the corresponding primes of $\QQ$. We denote by $\OKS$ the ring of $S$-integers of $K$, i.e. the ring of elements of $K$ that have non-negative valuation for all valuations $v$ on $K$ with $v \notin S$. For any finite place $v$, we choose some Frobenius $\Frob_v \in \Gal(\overline{K}/K)$. We denote by $K_v$ the $v$-adic completion of $K$. The unique continuous extension of $\Frob_v$ to $K_v$ will also be denoted $\Frob_v$. We will also fix an inclusion $K \subset \CC$. 

Let $X/K$ be a smooth, geometrically connected, but not necessarily proper, algebraic variety. We will assume that $X$ has a smooth model $\mathcal{X}$ over $\OKS$, and we will fix one. We are interested in studying the set $\mathcal{X}(\OKS)$ of $S$-integral points of this model, which by abuse of notation we will just denote by $X(\OKS)$. Note that in general this set depends on the choice of the model $\mathcal{X}$. 

\subsection*{Acknowledgments}
I am grateful to Jaap Top for interesting conversations and useful comments on this work.

\subsection*{Addendum}
Shortly after the first version of this work appeared on arXiv, A. Cadoret kindly alerted me to the master thesis of H. Liu \cite{LiuMasterThesis}, which contains, along with a very well-written and detailed explanation of the Lawrence--Venkatesh proof of Faltings' theorem, also a proof of Siegel's theorem via the Lawrence--Venkatesh method, although via a differently constructed abelian-by-finite family. 

\section{The Lawrence--Venkatesh framework}
In this section we explain the Lawrence--Venkatesh framework, in the case of an abelian-by-finite family. To simplify matters, we will use a theorem by Faltings (\cite{Faltings}) which says that geometric \'etale cohomology of abelian varieties is semisimple (the authors of \cite{Lawrence--Venkatesh} explicitly avoid using this theorem in order to increase independence of their proof of Faltings' theorem from Faltings' own proof). We start by recalling relevant definitions from \cite{Lawrence--Venkatesh}.

\section{The Lawrence--Venkatesh framework}
In this section we explain the Lawrence--Venkatesh framework, in the case of an abelian-by-finite family. To simplify matters, we will use a theorem by Faltings (\cite{Faltings}) which says that geometric \'etale cohomology of abelian varieties is semisimple (the authors of \cite{Lawrence--Venkatesh} explicitly avoid using this theorem in order to increase independence of their proof of Faltings' theorem from Faltings' own proof). We start by recalling relevant definitions from \cite{Lawrence--Venkatesh}.

The following is Definition 5.1 in \cite{Lawrence--Venkatesh}. 
\begin{definition}
An \emph{abelian-by-finite family over $X$} is a sequence of maps $A \overset{f}\to X' \overset{\pi}\to X$ where $A \to X'$ is a family of polarized abelian varieties and $X' \to X$ is finite \'etale. A \emph{good model} of such an abelian-by-finite family over $\OKS$ is a sequence of maps $\mathcal{A} \to \mathcal{X}' \to \mathcal{X}$ of schemes over $\OKS$, where again $\mathcal{A} \to \mathcal{X}'$ is a family of polarized abelian varieties and $\mathcal{X}' \to \mathcal{X}$ is finite \'etale, which give back the maps $A \to X' \to X$ after base change $- \otimes _{\OKS} K$, and where $\mathcal{X}$ is the model of $X$ that we fixed, and which moreover satisfies the following technical conditions: the cohomology sheaves $\mathbf{R}^q(\pi\!\circ\! f)_*\Omega^p_{\mathcal{A}/\mathcal{X}}$ and de Rham sheaves $\mathcal{H}^q = \mathbf{R}^q(\pi\!\circ\! f)_*\Omega^{\bullet}_{\mathcal{A}/\mathcal{X}}$ on $\mathcal{X}$ are locally free as $\OO_{\mathcal{X}}$-modules, and the Gauss-Manin connection extends to a morphism $\mathcal{H}^q \to \mathcal{H}^q \tensor \Omega^1_{\mathcal{X}/\OKS}$ on $\mathcal{X}$. 
\end{definition}

By the usual arguments, any abelian-by-finite family over $K$ has a good model over $\OK[(S')\inv]$ for some finite set $S' \supseteq S$ of places of $K$.  

The second important definition from \cite{Lawrence--Venkatesh} we need is the notion of full monodromy, which we now explain. Fix an abelian-by-finite family 
\[A \overset{f}\to X' \overset{\pi}\to X.\]
We denote by $g$ the relative dimension of $A$ over $X'$. With respect to the inclusion $K \subset \CC$ we fixed, we get continuous maps $A(\CC) \to X'(\CC) \to X(\CC)$. Fix a point $x_0 \in X(\CC)$. The fiber $\pi\inv(x_0) \subsetold X'(\CC)$ consists of $\deg \pi$ points, and we get a decomposition
\[ H^1_B(A_{x_0}(\CC), \QQ) = \bigoplus_{\pi(x') = x_0} H^1_B(A_{x'}(\CC), \QQ). \]
Here $H^1_B$ denotes the first Betti cohomology (also known as singular cohomology). The action of the fundamental group $\pi_1(X(\CC), x_0)$ by monodromy on these spaces preserves this decomposition (but not necessarily the individual factors). The symplectic form $\omega$ induced by the polarization on $A$ is also preserved by this action. 

\begin{definition}
Let $\rho\colon \pi_1(X(\CC), x_0) \to \GL(H^1_B(A_{x_0}(\CC), \QQ))$ be the monodromy representation. The abelian-by-finite family has \emph{full monodromy} if the Zariski closure of the image of $\rho$ contains $\prod_{\pi(x') = x_0} \operatorname{Sp}(H^1_B(A_{x'}(\CC), \QQ), \omega)$. 
\end{definition}

In order to apply the Lawrence--Venkatesh method, we need to make sure that the Frobenius action on the fibers of $\pi$ is large enough. In order to make this precise, we make the following definition, which is a variation of \cite[Definition 5.2]{Lawrence--Venkatesh} 
\begin{definition}
Let $E$ be a finite $\Gal(\overline{K}/K)$-set and $v$ a place of $K$ for which this action is unramified. Then the \emph{$v$-length of $E$}, denoted $\length_v(E)$, is the average size of the orbits in $E$ under the action of $\Frob_v$, i.e.
\[ \length_v(E) = \frac{\#E}{\textrm{number of $\Frob_v$-orbits}} \]
If $E$ is a finite $K$-scheme we write $\length_v(E)$ instead of $\length_v(E(\overline{K}))$. Note that the $v$-length of $E$ does not depend on the choice of the Frobenius element $\Frob_v$. 
\end{definition}

The following theorem is a compact restatement of the Lawrence--Venkatesh method. This theorem is not stated as such in \cite{Lawrence--Venkatesh}, but all the components of the proof are there. Therefore we will give a short sketch of the argument. We stress that our arguments uses the fact that the Galois representation on \'etale cohomology of an abelian variety over a number field is semi-simple, which is a deep theorem of Faltings \cite{Faltings}. The authors of \cite{Lawrence--Venkatesh} avoid using this theorem, both to keep their proof of Faltings' theorem independent of Faltings' own proof, and to be able to apply the same methods to more general families for which semisimplicity is not unconditionally known. We have also incorporated the results of Bakker and Tsimerman \cite{Bakker-Tsimerman}, as explained in \cite[Section 9]{Lawrence--Venkatesh}. 

\begin{theorem}\label{main-theorem}
Let $T \subset X(\OKS)$ be a subset of $S$-integral points of $X$. Assume that there is an abelian-by-finite family $A \overset{f}\to X' \overset{\pi}\to X$, with $\frac{1}{2}g(g+1)\deg(\pi) > \dim(X)$, with the following properties
\begin{enumerate}
    \item The family has full monodromy, and
    \item For every $x \in T$, there is a finite place $v$ of $K$ for which the Galois action on $\pi\inv(x)$ is unramified and such that 
    \begin{equation*} 
        \length_v(\pi\inv(x)) \geq \frac{ 4g^2 \deg(\pi)}{\frac{1}{2}g(g+1)\deg(\pi) - \dim(X)}.
    \end{equation*}
\end{enumerate}
Then $T$ is not Zariski dense in $X$. 
\end{theorem}
\begin{proof}[Proof sketch]
After possibly increasing $S$, we can and will assume that the abelian-by-finite family admits a good model over $\OKS$. Note also that in the second condition, we can restrict to $v \notin S$: for an $x \in T$ and $v$ as in the second condition, Chebotarev's density theorem implies that there are infinitely many places $v'$ of $K$ for which the Galois action on $\pi\inv(x)$ is also unramified, and for which the actions of $\Frob_v$ and $\Frob_{v'}$ are conjugate and hence give the same length. Moreover, the assumption that the abelian-by-finite family has a good model over $\OKS$ implies that the residue fields of the points in $\pi\inv(x)$ are field extensions of $K$ which are unramified away from $S$. It follows from the Hermite-Minkowski theorem that there are only finitely many possibilities for $\pi\inv(x)$ as a $\Gal(\overline{K}/K)$-set, where $x$ ranges over $T$. Therefore, we can choose the places $v$ that occur in condition (2) to lie in a fixed finite set of places $S'$ with $S \cap S' = \emptyset$. Therefore we may decompose $T = T_1 \cup \cdots \cup T_r$ as a finite union of subsets, where for each $T_i$ there is a single place $v \in S'$ that works for all $x \in T_i$. Since a finite union of non-Zariski dense sets is not Zariski dense, we may replace $T$ by one of the $T_i$, and assume that we have a fixed place $v \notin S$ that works for all points $x \in T$. Let us fix such a place $v$. 

Let $x \in T$. We have $\pi\inv(x) = \Spec E_x$ for a $K$-algebra $E_x$ of degree $\deg \pi$. As in \cite[Lemma 2.3]{Lawrence--Venkatesh}, there are only finitely many possibilities for the filtered $\phi$-module $V_x := \Hdr^1(A_x / K_v)$ (as mentioned above, the \'etale cohomologies are semisimple so we don't need to bother with semisimplification). Through the factorization \[A_{x,v} \to \Spec E_{x,v} \to \Spec K_v,\]
where $A_{x,v} := A_x \tensor_K K_v$ and $E_{x,v} = E_x \tensor_K K_v$, this filtered $\phi$-module moreover has the structure of a free $E_{x,v}$-module of rank $2g$. The Frobenius of $V_x$ is compatible with the Frobenius of $E_{x,v}$, and the filtration of $V_x$ is given by a free and saturated $E_{x,v}$-submodule, which is moreover Lagrangian with respect to the symplectic form $\omega$ on $V_x$. For any point $x' \in X(K_v)$ with $x \equiv x' \mod v$, the same statements are true for $E_{x'}$ (where $\pi\inv(x') = \spec E_{x'}$) and $V_{x'} := \Hdr^1(A_{x'} / K_v)$. The Gauss-Manin connection gives us canonical bijections $E_{x,v} \overset{\sim}\to E_{x'}$ and $V_{x} \overset{\sim}\to V_{x'}$. These bijections commute with all the mentioned structure except for the filtration. Variation of this filtration then gives the $p$-adic period map
\[ \Phi_v\colon U \to \mathcal{H}_v(K_v) \]
where $U = \{x' \in X(K_v) : x' \equiv x \pmod v\}$ and $\mathcal{H}_v = \Res_{K_v}^{E_{x,v}} \operatorname{LGr}(V_x, \omega)$. Here $\operatorname{LGr}(V_x, \omega)$ is the Lagrangian Grassmannian classifying free rank-$g$ $E_{x,v}$-submodules of $V_x$ on which $\omega$ is trivial, and $\Res_{K_v}^{E_{x,v}}$ denotes Weil-restriction from $E_{x,v}$ to $K_v$. So $\mathcal{H}_v$ is an algebraic variety over $K_v$ whose $K_v$-points classify free rank-$g$ $E_{x,v}$-submodules of $V_x$ on which $\omega$ is trivial. The Zariski dimension of $\mathcal{H}_v$ is $\frac{1}{2}g(g+1)\cdot \deg(\pi)$. From \cite[Lemma 3.3]{Lawrence--Venkatesh}, the full monodromy assumption implies that the period map has Zariski-dense image.

Let $\phi$ be the Frobenius on $V_x$. We may write $E_{x,v} = L_1 \times \cdots \times L_r$ for unramified field extensions $L_i / K_v$. Note that $r$ is the number of $\Frob_v$-orbits on $\pi\inv(x)$, so that $\length_v(\pi\inv(x)) = \deg \pi / r$. Then the assumption on $\length_v(E_{x})$ means that 
\[ r \leq \frac{\frac{1}{2}g(g+1)\deg(\pi) - \dim(X)}{4g^2}.\]

The $E_{v,x}$-module on $V_x$ induces a decomposition $V_x = \bigoplus_{i=1}^r V_i$, where $V_i = V_x \tensor_{E_{v,x}} L_i$. Each $V_i$ is a $L_i$-vector space of dimension $2g$.  We write 
\[Z(\phi) = \{f \colon V_x \to V_x \textrm{ $E$-linear, and } \phi \circ f = f\circ \phi\}.\]
Then $Z(\phi)$ is a $K_v$-vector space. By a variation of \cite[Lemma 2.1]{Lawrence--Venkatesh} we have 
\[ \dim_{K_v} ( Z(\phi) ) \leq \sum_{i = 1}^r (\dim_{L_i} V_i)^2 = 4g^2 \cdot r \leq \frac{1}{2}g(g+1)\deg(\pi) - \dim(X).\]
Set $T' = T \cap U$. As in \cite[Section 3]{Lawrence--Venkatesh} it follows that $\Phi_v(T')$ is contained in an algebraic subset of $\mathcal{H}_v$ of the form $\mathcal{Z} := \bigcup_{i} Z(\phi) \cdot h_i$, for finitely many $h_i \in \mathcal{H}_v(K_v)$. We have $\dim \mathcal{Z} \leq \frac{1}{2}g(g+1)\deg(\pi) - \dim(X)$ and $\dim \mathcal{H}_v = \frac{1}{2}g(g+1)\deg(\pi)$, so 
\[\operatorname{codim}_{\mathcal{H}_v} \mathcal{Z} \geq \dim(X).\]
Applying \cite[Lemma 9.3]{Lawrence--Venkatesh}, we find that $\Phi_v\inv(\mathcal{Z})$, and therefore $T'$, is not Zariski dense in $X$. We do this for all of the finitely many $v$-adic residue disks of $X(K_v)$, and since the finite union of non-Zariski dense sets is not Zariski dense, we conclude that $T$ is not Zariski dense in $X$.  
\end{proof}

\section{An abelian-by-finite cover of \texorpdfstring{$E\setminus\{0\}$}{E\{0\}}}

In what follows, we fix an elliptic curve $E / K$ which has good reduction away from $S$.
Our goal is to construct an abelian-by-finite cover on $E \setminus\{0\}$ that proves Siegel's theorem: the set of $S$-integral points of $E$ is finite.  

In order to prove Siegel's theorem we may increase $K$ and $S$ without loss of generality (the set of $S$-integral points of $E$ will only increase if we do so). Hence, we may and will assume in what follows that $S$ contains all primes over $2$, and that the 2-torsion of $E$ is $K$-rational, and that the model of $E$ that we fixed over $O_K[S\inv]$ is given by the Weierstrass equation
\[ y^2 = x(x-1)(x-\lambda) \]
for some $\lambda \in \O_K[S\inv]$. As in the introduction, we will also continue to assume that $S$ contains all places of $K$ that are ramified in the extension $K/\QQ$. 

To define the abelian-by-finite family over $E \setminus \{0\}$, we first consider the family $A'$ of elliptic curves over $E \setminus E[2]$ defined by the equation
\[ v^2 = u (u-1)(u - x). \]
Here $x$ is still the $x$-coordinate on $E$, viewed as a regular function on $E \setminus E[2]$. In other words $A'$ is the pull-back of the Legendre family over $\PP^1\setminus\{0,1,\infty\}$ along the $x$-coordinate map $E\setminus E[2] \to \PP^1\setminus\{0,1,\lambda, \infty\}$.
For every $m \geq 1$ we define the family $A_m$ over $E\setminus \{0\}$ as the composition of the map $A' \to E\setminus E[2]$, restricted to $E \setminus E[2^m]$, with the multiplication-by-$2^m$ map $[2^m]\colon E\setminus E[2^m] \to E\setminus\{0\}$. Geometrically, the fiber of $A_m$ over a geometric point $e \in E\setminus \{0\}$ is the disjoint union of the elliptic curves $v^2=u(u-1)(u-x')$ where $x'$ runs over the $x$-coordinates of the $2^{2m}$ geometric points of $E$ mapping to $e$ under the $[2^m]$ map. The situation is depicted in the following diagram:

\begin{center}
    \begin{tikzpicture}
      \matrix (m) [matrix of math nodes,row sep=3em,column sep=3em,minimum width=3em]
      {
        \mathrm{Legendre} & A_m &  \\
        \PP^1\setminus D & E \setminus E[2^m] & E \setminus \{0\}\\};
      \path[-stealth]
        (m-1-2) edge node [left]  {$\pi$} (m-2-2)
        (m-2-2) edge node [above] {$[2^m]$} (m-2-3)
        (m-1-2) edge node [above] {} (m-1-1)
        (m-2-2) edge node [above] {$x$} (m-2-1)
        (m-1-1) edge node [left] {} (m-2-1);
    \end{tikzpicture} 
\end{center}
Here $D$ is the divisor of $x$-coordinates of points in $E[2^m]$ (including $\infty$) and $\mathrm{Legendre}$ denotes the Legendre family $v^2 = u(u-1)(u-t)$. To prove Siegel's theorem we will apply \cref{main-theorem} to the abelian-by-finite family $A_m \to E \setminus E[2^m] \to E \setminus \{0\}$ for $m = 3$. 

\subsection{The monodromy}

Let $e_0 \in E(K) \setminus\{0\}$ be some arbitrary base point. Let $W_m = H^1_{B}((A_m)_{e_0}(\CC), \QQ)$ be the first Betti cohomology of the fiber of $A_m \to E \setminus \{0\}$ over $e_0$. Monodromy gives a representation 
\[ \rho_m \colon \pi_1(E(\mathbb{C}) \setminus \{0\}, e_0) \longrightarrow \operatorname{GL}(W_m). \]
Our goal in this subsection is to study this representation. To ease the notation, we will abuse notation in this subsection and identify varieties with their $\CC$-points, i.e. we will write $E$ instead of $E(\CC)$. 

Note that we have
\[ W_m  = \bigoplus_{[2^m]e' = e_0} H^1_{B}((A_m)_{e'}, \mathbb{Q}), \] 
where the direct sum runs over all points $e' \in E\setminus E[2^m]$ mapping to $e_0$ under $[2^m]$. Note that the fiber $(A_m)_{e'}$ is just the elliptic curve given by $v^2 = u(u-1)(u-x')$, where $e' = (x',y')$. 

The following proposition says that the family $A_m \to E\setminus\{0\}$ has full monodromy. 
\begin{proposition}\label{big-monodromy}
The Zariski closure of the image of $\rho_m$ in $\GL(W_m)$ contains 
\[\prod_{[2^m]e' = e_0} \SL(H^1_{B}((A_m)_{e'}, \QQ)).\]
\end{proposition}
\begin{proof}
Let $\Gamma$ be the Zariski closure of the image of $\rho$ in $\GL(W_m)$, and set
\[\Delta = \Gamma \cap \prod_{[2^m]e' = e_0} \SL(H^1_{B}((A_m)_{e'}, \mathbb{Q})).\]
We will use \cite[Lemma 2.12]{Lawrence--Venkatesh}, which says the following: if $\Delta$ projects surjectively to each factor $\SL(H^1_{B}((A_m)_{e'}, \mathbb{Q}))$, and if for each pair $e'_1 \neq e'_2$ there is a $g \in \Delta$ which projects to unipotent elements of $\SL(H^1_{B}((A_m)_{e'_1}, \mathbb{Q}))$ and $\SL(H^1_{B}((A_m)_{e'_2}, \mathbb{Q}))$ with fixed spaces of different dimensions, then in fact $\Delta = \prod_{[2^m]e' = e_0} \SL(H^1_{B}((A_m)_{e'}, \mathbb{Q}))$.  

First we consider surjectivity of the projections. Let $e'_0 \in E \setminus E[2^m]$ be a point with $[2^m]e'_0 = e_0$. Then the \'etale map $[2^m]\colon E \setminus E[2^m] \to E\setminus \{0\}$ induces on fundamental groups an injective map $\pi_1(E \setminus E[2^m], e_0') \to \pi_1(E \setminus \{0\}, e_0)$, and the restriction of $\rho$ to $\pi_1(E \setminus E[2^m], e_0')$ stabilizes the direct summand $H^1_{B}((A_m)_{e'_0}, \mathbb{Q})$ of $W_m$. 
But we also can consider the $x$-coordinate map $x\colon E\setminus E[2^m] \to \PP^1 \setminus D$, where $D$ is the set of $x$-coordinates of the $2^m$-torsion points of $E$. This map is also \'etale, and it induces an injective homomorphism $\pi_1(E \setminus E[2^m], e_0') \to \pi_1(\PP^1 \setminus D, x')$. The cover $A_m \to E \setminus E[2^m]$ is just the pull-back of the Legendre family over $\PP^1\setminus D$, and the Zariski closure of the image of monodromy for the Legendre family is well-known to be  $\SL_2$ (see for example \cite[Theorem 1.1.7]{PeriodBook} for an explicit description of this monodromy). So the restriction of $\rho$ to $\pi_1(E \setminus E[2^m], e_0')$ has Zariski-closure a finite index subgroup of $SL_2$, which is therefore equal to $SL_2$ since $SL_2$ is connected.    

Now we show that there are enough elements acting unipotently on the various summands with fixed spaces of different dimensions. Let $\gamma$ be a path in $E \setminus \{0\}$ that goes from $e_0$ to a point close to the identity $0 \in E$, then circles around this point once, and goes back again to $e_0$ in the same way. A lift of this path to $E \setminus E[2^m]$ is a loop from a point $e'_0$ over $e_0$ going around a $2^m$-torsion point and back to $e'_0$. The family $A_m$ extends in a smooth way over all $2^m$-torsion points except the points $(0,0)$, $(1,0)$ and $0$. Over these points, the monodromy is unipotent: indeed, the local monodromy of the Legendre family over $\PP^1\setminus\{0,1,\infty\}$ is unipotent around $0$ and $1$ and conjugate to \scalebox{0.7}{$\begin{pmatrix} -1 & 1 \\ 0 & -1 \end{pmatrix}$} around $\infty$, but the $x$-coordinate is locally 2-to-1 around these points. So we see that $\rho(\gamma)$ fixes each summand of the decomposition
$W_m = \bigoplus_{[2^m]e' = e_0} H^1_{B}((A_m)_{e'}, \mathbb{Q})$, and  
it acts in a non-trivial unipotent way on exactly three of the summands and trivially on the other summands. The three summands on which $\rho(\gamma)$ acts non-trivially depend on the choice of $\gamma$; more specifically on the choice of the path from $e_0$ to the point close to $0$. Now let $e_1'$ and $e_2'$ be two distinct points in the fiber $[2^m]\inv(e_0)$. We claim that we can choose the path $\gamma$ in such a way that $\rho(\gamma)$ acts non-trivially on the summand corresponding to the fiber over $e_1'$ and trivially on the summand corresponding to the fiber over $e_2'$. For this, let $d = e_2' - e_1' \in E$ be the difference. Then $d$ is $2^m$-torsion and non-trivial. The set $\{(0,0), (1,0), 0\}$ is not closed under shifting by $d$. Pick a point $P \in \{(0,0), (1,0), 0\}$ such that $P + d \notin \{(0,0), (1,0), 0\}$. Then choose a path $\gamma'$ in $E \setminus E[2^m]$ that starts at $e_1'$, goes to a point close to $P$, circles around $P$, and then goes back to $e_1'$ the same way. Set $\gamma = [2^m]\circ \gamma'$. This is a loop in $E \setminus \{0\}$ based at $e_0$. By construction, $\gamma'$ is the lift of $\gamma$ starting at $e_1'$, and it loops around the bad fiber over $P$, while $\gamma' + d$ is the lift of $\gamma$ starting at $e_2'$. The latter loops around $P + d$, which is not a bad fiber. Thus $\rho(\gamma)$ acts non-trivially unipotently on the summand corresponding to $e_1'$ and trivially on the summand corresponding to $e_2'$. 
\end{proof}

\subsection{Siegel's theorem}
We can now prove Siegel's theorem. We follow the argument in \cite[Section 4]{Lawrence--Venkatesh}. 

First we note that we may assume, after enlarging $K$ if necessary, that the $8$-torsion of $E$ is $K$-rational. Now let
\[ T = \{ e \in (E\setminus\{0\})(\OO_K[S\inv]) : e \notin 2E(K)\} \]
be the set of $S$-integral points of $E$ that are not divisible by $2$. We claim that it is enough to prove that $T$ is finite. Indeed, let $k$ be the largest integer such that $E(K)$ contains a point of order $2^k$. Then we have 
\[ (E\setminus\{0\})(\OO_K[S\inv]) \,\subseteq\, \bigcup_{j = 0}^k [2^j](T). \]
To see this, note that if $e$ is an $S$-integral point of $E$ and $e'$ a rational point such that $e = 2e'$, then also $e'$ is $S$-integral (because $e'$ does not reduce to the identity modulo any place $v \notin S$). If an $S$-integral point $e$ of $E$ is not divisible by $2^k$, then $e$ is in the right hand side of the claimed inclusion. Otherwise, write $e = 2^k \cdot e'$. Then by adjusting $e'$ by a $2^k$-torsion point if needed, we can ensure that $e' \in T$, so that $e \in [2^k](T)$. 

In order to prove finiteness of $T$, we want to apply \cref{main-theorem} to $T$ and the abelian-by-finite family $A_3 \to E \setminus E[8] \to E\setminus\{0\}$. This will imply that $T$ is not Zariski-dense in $E$ and therefore finite. We have already established that this family has full monodromy, so it remains to verify that condition 2 of the theorem is fulfilled. 

\begin{lemma}
Let $t \in T$. There is a place $v \notin S$ of $K$ such that all $\Frob_v$-orbits in $[8]\inv(t)(\overline{K})$ have length $8$. In particular $\length_v([8]\inv(t)) = 8$. 
\end{lemma}
\begin{proof}
Write $[8]\inv(t) = \Spec E_t$. Then we have $E_t = \prod_{i} L_i$, where the $L_i$ are finite field extensions of $K$ that are unramified away from $S$. In fact, the $L_i$ are the fields obtained by adjoining the coordinates of points in $[8]\inv(t)$ to $K$. Since the $8$-torsion of $E$ is $K$-rational, these fields are in fact the same, so we have $E_t = L^n$ for some $n$. 

To study the Galois action on the factors $L$ in this decomposition, let $\sigma \in \Gal(\Kbar/K)$. Then for each $t' \in E(\Kbar)$ with $8t' = t$, we also have $8\sigma(t') = t$. Therefore we have $\sigma(t') - t' \in E[8](\overline{K})$. Using that the $8$-torsion is $K$-rational, it is not hard to see that this element $\sigma(t') - t'$ does not depend on the choice of $t'$. Thus we get a map $\Gal(\Kbar/K) \to E[8](K) \cong (\ZZ/8\ZZ)^2$ that describes the Galois action on the fiber over $t$. (This is of course just a well-known explicit description of the Kummer map $E(K) \to H^1(\Gal(\Kbar/K), E[8]) = \Hom(\Gal(\Kbar/K), E[8](K))$, where the last equality follows from the 8-torsion being $K$-rational.) The image of $\Gal(\Kbar/K) \to E[8](K)$ contains an element of order 8: if it didn't, then $\Gal(\Kbar/K)$ would stabilize the point $4t'$, contradicting the assumption that $t$ is not divisible by $2$ in $E(K)$. Thus, there is some $\sigma \in \Gal(\Kbar/K)$ which acts with order 8 on some (hence each) factor $L$. By Chebotarev's density theorem, there is a positive density of places $v \notin S$ such that $\Frob_v$ acts in the same way as $\sigma$ on $L$, and then $\Frob_v$ has orbits of length $8$ on $[8]\inv(t)(\overline{K})$. 
\end{proof}

\begin{corollary}[Siegel's theorem]
The set $(E\setminus \{0\})(\OKS)$ of $S$-integral points of $E$ is finite. 
\end{corollary}

\begin{remark}
The proof of Siegel's theorem presented in this section relies on \cref{main-theorem}, which in turn relies on Faltings' deep result that the the Galois representations attached to abelian varieties over number fields are semisimple. However, for the purpose of proving Siegel's theorem the dependence on Faltings' result can be removed, in the same way as for the Lawrence--Venkatesh proof of the $S$-unit theorem, by applying \cite[Lemma 4.4]{Lawrence--Venkatesh}. 
\end{remark}

\bibliographystyle{alphaurl}
\bibliography{biblio.bib}

\newcommand{\etalchar}[1]{$^{#1}$}
\begin{thebibliography}{BBB{\etalchar{+}}19}

\bibitem[BBB{\etalchar{+}}19]{ManyAuthors}
Jennifer~S. {Balakrishnan}, Alex~J. {Best}, Francesca {Bianchi}, Brian
  {Lawrence}, J.~{Steffen M{\"u}ller}, Nicholas {Triantafillou}, and Jan
  {Vonk}.
\newblock {Two recent p-adic approaches towards the (effective) Mordell
  conjecture}.
\newblock {\em arXiv e-prints}, page arXiv:1910.12755, October 2019.
\newblock \href {http://arxiv.org/abs/1910.12755} {\path{arXiv:1910.12755}}.

\bibitem[BT19]{Bakker-Tsimerman}
Benjamin Bakker and Jacob Tsimerman.
\newblock The {A}x-{S}chanuel conjecture for variations of {H}odge structures.
\newblock {\em Invent. Math.}, 217(1):77--94, 2019.
\newblock \href {https://doi.org/10.1007/s00222-019-00863-8}
  {\path{doi:10.1007/s00222-019-00863-8}}.

\bibitem[CMSP17]{PeriodBook}
James Carlson, Stefan Müller-Stach, and Chris Peters.
\newblock {\em Period Mappings and Period Domains}.
\newblock Cambridge University Press, {S}econd edition, 2017.

\bibitem[Cor16]{Corvaja}
Pietro Corvaja.
\newblock {\em Integral points on algebraic varieties}, volume~3 of {\em
  Institute of Mathematical Sciences Lecture Notes}.
\newblock Hindustan Book Agency, New Delhi, 2016.
\newblock An introduction to Diophantine geometry.
\newblock \href {https://doi.org/10.1007/978-981-10-2648-5}
  {\path{doi:10.1007/978-981-10-2648-5}}.

\bibitem[Fal83]{Faltings}
G.~Faltings.
\newblock Endlichkeitss\"{a}tze f\"{u}r abelsche {V}ariet\"{a}ten \"{u}ber
  {Z}ahlk\"{o}rpern.
\newblock {\em Invent. Math.}, 73(3):349--366, 1983.
\newblock \href {https://doi.org/10.1007/BF01388432}
  {\path{doi:10.1007/BF01388432}}.

\bibitem[JL19]{Javanpeykar-Litt}
Ariyan {Javanpeykar} and Daniel {Litt}.
\newblock {Integral points on algebraic subvarieties of period domains: from
  number fields to finitely generated fields}.
\newblock {\em arXiv e-prints}, page arXiv:1907.13536, July 2019.
\newblock \href {http://arxiv.org/abs/1907.13536} {\path{arXiv:1907.13536}}.

\bibitem[Liu20]{LiuMasterThesis}
Haohao Liu.
\newblock Lawrence--{V}enkatesh's $p$-adic approach to {M}ordell's conjecture.
\newblock Master thesis, 2020.
\newblock URL: \url{http://home.ustc.edu.cn/~kyung/Lawrence-Venkatesh.pdf}.

\bibitem[LS20]{Lawrence-Sawin}
Brian {Lawrence} and Will {Sawin}.
\newblock {The Shafarevich conjecture for hypersurfaces in abelian varieties}.
\newblock {\em arXiv e-prints}, page arXiv:2004.09046, April 2020.
\newblock \href {http://arxiv.org/abs/2004.09046} {\path{arXiv:2004.09046}}.

\bibitem[LV20]{Lawrence--Venkatesh}
Brian Lawrence and Akshay Venkatesh.
\newblock Diophantine problems and {$p$}-adic period mappings.
\newblock {\em Invent. Math.}, 221(3):893--999, 2020.
\newblock \href {https://doi.org/10.1007/s00222-020-00966-7}
  {\path{doi:10.1007/s00222-020-00966-7}}.

\bibitem[NX19]{Nasserden-Xiao}
Brett {Nasserden} and Stanley~Yao {Xiao}.
\newblock {Uniformity of fibres of period mappings and the $S$-unit equation}.
\newblock {\em arXiv e-prints}, page arXiv:1910.14122, October 2019.
\newblock \href {http://arxiv.org/abs/1910.14122} {\path{arXiv:1910.14122}}.

\end{thebibliography}

\end{document}